\documentclass[leqno,12pt]{amsart}

\usepackage{amssymb} 
\usepackage{amsmath}
\usepackage{amsthm}
\usepackage{amscd}
\usepackage{bm}
\usepackage{graphicx,psfrag,wrapfig}

\theoremstyle{plain} 
\newtheorem{theorem}{Theorem}[section] 
\newtheorem{lemma}[theorem]{Lemma}
\newtheorem{corollary}[theorem]{Corollary}
\newtheorem{proposition}[theorem]{Proposition}

\theoremstyle{definition} 

\newtheorem{remark}[theorem]{Remark}
\newtheorem{example}[theorem]{Example}

\newtheorem{problem}[theorem]{Problem}
\newcommand{\affine}{\mathbb{C}}
\newcommand{\Cr}{\mathrm{Cr}}

\newcommand{\rank}{\mathrm{rank}}

\newcommand{\norm}[1]{\left|\!\left|#1\right|\!\right|}
\newcommand{\hess}{\mathrm{Hess}}

\def\RR{\mathbb{R}}
\def\ZZ{\mathbb{Z}}
\def\TT{\mathbb{T}}

\def\cU{{\mathcal U}}
\def\cV{{\mathcal V}}
\def\cD{{\mathcal D}}
\def\cH{{\mathcal H}}

\def\cT{{\mathcal T}}

\def\ra{{\rightarrow}}
\def\del{\partial}

\def\ev{\text{ev}}

\DeclareMathOperator{\Diff}{Diff}
\DeclareMathOperator{\Int}{Int}
\DeclareMathOperator{\Id}{Id}

\newcommand\cl[1]{{\overline{#1}}}

\begin{document}

\title[Growth of critical points]
{Growth of critical points in one-dimensional lattice systems} 

\author[M. Asaoka, T. Fukaya, K. Mitsui, M. Tsukamoto]
{Masayuki Asaoka, Tomohiro Fukaya, Kentaro Mitsui, Masaki Tsukamoto}

\begin{abstract}
We study the growth of the numbers of critical points in 
one-dimensional lattice systems
by using (real) algebraic geometry and the theory of 
homoclinic tangency.
\end{abstract}



\date{\today}


\maketitle

\section{Introduction: One-dimensional lattice system}

Let $M$ be a compact connected $C^\infty$ manifold without boundary.
Let $f:M\times M\to \mathbb{R}$ be a $C^\infty$ function.
For positive integers $n$, we define $f_n:M^{n+1}\to \mathbb{R}$ by setting 
\begin{equation} \label{eq: definition of f_n}
 f_n(p_1, p_2, \dots, p_{n+1}) := \sum_{i=1}^n f(p_i, p_{i+1}).
\end{equation}
Bertelson-Gromov \cite{Bertelson-Gromov} proposed the study of this kind of functions.
(See also Bertelson \cite{Bertelson}.)
Let $\mathrm{Cr}(f_n)$ be the set of critical points of $f_n$.
We are interested in the asymptotic behavior of $\Cr(f_n)$ as $n\to \infty$.
(The paper \cite{Fukaya-Tsukamoto} studies an asymptotic behavior of critical values of 
$f_n$.)

Naively speaking, the study of $f_n$ is a model of a 1-dimensional crystal (lattice system) 
which consists of $n+1$ particles.
We assume that the manifold $M$ is the configuration space of each particle and that 
$f(x,y)$ is the potential function describing the interaction between two adjacent particles.
Then $f_n$ is the total potential energy of the system,
and the critical points of $f_n$ are its stationary states.
 
Our view point and methods are motivated by the works of Artin-Mazur \cite{Artin-Mazur} and 
Kaloshin \cite{Kaloshin 1, Kaloshin 2}.
They studied the growth of periodic points of diffeomorphisms of manifolds
by using (real) algebraic geometry and the theory of homoclinic tangency.
We develop analogous methods for the study of $\mathrm{Cr}(f_n)$.

Let $C^\infty(M\times M)$ be the space of real valued $C^\infty$ functions in $M\times M$ 
with the $C^\infty$ topology.
Our first main result is the following theorem.
(Recall that a smooth function on a manifold is called a Morse function if all its critical points
are non-degenerate.)
\begin{theorem} \label{thm: main theorem}
There exists a dense subset $\mathcal{D} \subset C^\infty(M\times M)$ such that every 
$f\in \mathcal{D}$ satisfies the following two conditions:

\noindent 
(i) For all positive integers $n$, the functions $f_n:M^{n+1}\to \mathbb{R}$ are Morse functions.

\noindent 
(ii) There exists a positive real number $d$ (which depends on $f$) such that for all $n\geq 1$
\[ \#\Cr (f_n) \leq d^n.\]
Here $\#\Cr(f_n)$ is the number of the critical points of $f_n$.
\end{theorem}
From the condition (i) in Theorem \ref{thm: main theorem} and the Morse inequality 
(\cite{Milnor}), we have 
\[ \#\Cr(f_n) \geq \sum_{k\geq 0} \dim (H^k(M^{n+1};\mathbb{Z}_2)) = \left(\sum_{k\geq 0} \dim H^k(M;\mathbb{Z}_2)\right)^{n+1}\]
for all $f\in \mathcal{D}$.
Hence, from the condition (ii), 
\[  \left(\sum_{k\geq 0} \dim H^k(M;\mathbb{Z}_2)\right)^{n+1} 
  \leq \#\Cr(f_n) \leq d^n \quad (f\in \mathcal{D}, n\geq 1).\]
Therefore, if $\dim M\geq 1$, $\#\Cr(f_n)$ has an exponential growth for every $f\in \mathcal{D}$.
\begin{remark}
For each $n\geq 1$, the condition that $f_n$ is a Morse function is an open condition for 
$f\in C^\infty(M\times M)$.
Therefore Theorem \ref{thm: main theorem} (condition (i)) implies that the set 
\[ \{f\in C^\infty(M\times M)|\, \text{All $f_n$ $(n\geq 1)$ are Morse functions}\} \]
is a residual subset of $C^\infty(M\times M)$.
(Recall that a subset of a topological space is said to be residual if 
it contains a countable intersection of open dense subsets.)
This fact was already proved in Asaoka-Fukaya-Tsukamoto \cite[Theorem 1.2]{Asaoka-Fukaya-Tsukamoto}.
The argument in the present paper is totally different from that in \cite{Asaoka-Fukaya-Tsukamoto}.
The argument in \cite{Asaoka-Fukaya-Tsukamoto} is much more elementary.
It uses only elementary results in differential topology.
On the other hand, the argument of the 
present paper uses two very big theorems: Nash-Tognoli-King's theorem in real algebraic geometry 
(\cite{Nash}, \cite{Tognoli}, \cite{King}) and Hironaka's resolution of singularities (\cite{Hironaka}).
The important point is that we can achieve the condition (ii) in Theorem \ref{thm: main theorem}.
This is the new point of the present paper.
\end{remark}

By the above remark, the sets $\Cr(f_n)$ are finite sets for generic (i.e. residual) $f:M\times M\to \mathbb{R}$.
Theorem \ref{thm: main theorem} shows a regular behavior of 
$\mathrm{Cr}(f_n)$ for ``many'' (i.e. dense) $f$.
Next we will shows that there exists a very wild phenomenon.
We concentrate on the case $M=S^1 := \RR/2\pi\ZZ$.
\begin{theorem}
\label{thm:growth}
There exists a non-empty open subset $\cU$ of $C^\infty(S^1 \times S^1)$
 such that
 the set
\begin{equation*}
\left\{f \in \cU \;\left|\;
 \limsup_{n \ra \infty} \frac{\#\Cr(f_n)}{a_n} \geq 1
 \right. \right\}
\end{equation*}
 is residual in $\cU$
 for any given sequence $(a_n)_{n \geq 1}$ of positive integers.
\end{theorem}
For example, this implies that the set 
\[ \left\{f \in \cU \;\left|\;
 \limsup_{n \ra \infty} \frac{\#\Cr(f_n)}{\exp(\exp(n))} \geq 1
 \right. \right\} \]
is residual in the above given open set $\mathcal{U}$.
This shows a drastically unstable behavior of $\Cr(f_n)$ over $f\in \mathcal{U}$.
(Note that $\#\Cr(f_n)$ has an exponential growth on a dense subset of $\mathcal{U}$.)

It is interesting to see that there also exists a ``stable region'':
\begin{theorem} \label{thm: stable region}
There exist $d>0$ and a non-empty open set $\mathcal{V} \subset \mathcal{C}^\infty(S^1\times S^1)$
such that every $f\in \mathcal{V}$ satisfies the following conditions:

\noindent 
(i) All $f_n: (S^1)^{n+1}\to \mathbb{R}$ $(n\geq 1)$ are Morse functions.

\noindent 
(ii) For all $n\geq 1$, $\#\mathrm{Cr}(f_n) = d^{n+1}$.
\end{theorem}

The following question seems interesting.
\begin{problem}
Find a characterization of a function $f:M\times M\to \mathbb{R}$ which admits a neighborhood $\mathcal{W}
\subset C^\infty(M\times M)$ such that for every $g\in \mathcal{W}$ all functions 
$g_n:M^{n+1}\to \mathbb{R}$ $(n\geq 1)$ are Morse functions.
\end{problem}

\textbf{Acknowledgement.}
The authors wish to thank Professor Masahiro Shiota.
He kindly explained the proof of Proposition \ref{prop: algebraic model of our manifold, preliminary}
to the authors.
M. Asaoka, T. Fukaya and M. Tsukamoto were supported by Grant-in-Aid for Young Scientists (A) 
(22684003), (B) (23740049) and (B) (21740048)
respectively from JSPS.
K. Mitsui was supported by  Grant-in-Aid for JSPS Fellows (21-1111) from JSPS.

\section{Proof of Theorem \ref{thm: main theorem}} \label{section: proof of Theorem}
The essential ingredient of the proof of Theorem \ref{thm: main theorem}
is algebraic geometry.
As far as the authors know, the idea to use Nash's theorem and 
algebraic geometry in the study of $C^\infty$ manifolds 
goes back to Artin-Mazur \cite{Artin-Mazur}.
(For recent results on the Artin-Mazur type problem, see Kaloshin \cite{Kaloshin 1, Kaloshin 2}.)

\subsection{Preliminary}
Let $M$ be a compact connected $C^\infty$ manifold without boundary.
The next proposition will be proved later (Section \ref{section:  Proof of Proposition (algebraic model)}).
\begin{proposition}  \label{prop: algebraic model of our manifold}
There exist homogeneous polynomials (of real coefficients)
$F_i(X_0,X_1, \dots, X_d)$ $(1\leq i\leq R)$ in $\mathbb{R}[X_0,X_1,\dots,X_d]$ satisfying the following 
three conditions.

\noindent 
(i) The scheme $\mathrm{Proj}\left(\affine[X_0,X_1,\dots,X_d]/(F_1, \dots,F_R)\right)$ is an 
equidimensional regular scheme. 
(Set $\dim \mathrm{Proj}\left(\affine[X_0,X_1,\dots,X_d]/(F_1, \dots,F_R)\right) = d-r$.)
This implies the following:

Set $X :=\{[X_0:\cdots:X_d]\in \mathbb{P}^d(\affine)|\, F_i(X_0, \dots, X_d) =0 \> (1\leq i\leq R)\}$.
Then $X$ is a complex submanifold of $\mathbb{P}^d(\affine)$, and the (complex) dimension of every connected component of 
$X$ is equal to $d-r$. Moreover
\begin{equation*}
 \mathrm{rank}
 \begin{bmatrix} 
  \frac{\partial F_1}{\partial X_0} & \frac{\partial F_1}{\partial X_1} &\cdots & \frac{\partial F_1}{\partial X_d} \\
  \frac{\partial F_2}{\partial X_0} & \frac{\partial F_2}{\partial X_1} &\cdots & \frac{\partial F_2}{\partial X_d} \\
  \vdots & \vdots & \ddots & \vdots \\
  \frac{\partial F_R}{\partial X_0} & \frac{\partial F_R}{\partial X_1} &\cdots & \frac{\partial F_R}{\partial X_d} 
 \end{bmatrix}
 = r \quad \text{on $X$}.
\end{equation*}

\noindent 
(ii) $X$ transversally intersects with the hyperplane $\{X_0=0\}$ in $\mathbb{P}^d(\affine)$.
This implies the following:
Set $X_\infty := X\cap \{X_0=0\}$. Then
\begin{equation*}
\mathrm{rank}
 \begin{bmatrix}
 1 & 0 & \cdots & 0 \\
 \frac{\partial F_1}{\partial X_0} & \frac{\partial F_1}{\partial X_1} &\cdots & \frac{\partial F_1}{\partial X_d} \\
  \frac{\partial F_2}{\partial X_0} & \frac{\partial F_2}{\partial X_1} &\cdots & \frac{\partial F_2}{\partial X_d} \\
  \vdots & \vdots & \ddots & \vdots \\
  \frac{\partial F_R}{\partial X_0} & \frac{\partial F_R}{\partial X_1} &\cdots & \frac{\partial F_R}{\partial X_d} 
 \end{bmatrix}
 = 1 + \mathrm{rank}
 \begin{bmatrix}
  \frac{\partial F_1}{\partial X_1} &\cdots & \frac{\partial F_1}{\partial X_d} \\
  \frac{\partial F_2}{\partial X_1} &\cdots & \frac{\partial F_2}{\partial X_d} \\
  \vdots & \ddots & \vdots \\
  \frac{\partial F_R}{\partial X_1} &\cdots & \frac{\partial F_R}{\partial X_d} 
 \end{bmatrix}
 = 1 + r \quad \text{on $X_\infty$}.
\end{equation*}

\noindent 
(iii) Let $\mathbb{R}^d := \{[1:x_1:x_2:\cdots:x_d]\in \mathbb{P}^d(\affine)|\, x_1, \dots,x_d\in \mathbb{R}\}$.
Then $\underline{X}^\mathbb{R} := X\cap \mathbb{R}^d$ is diffeomorphic to $M$.
Here $\underline{X}^{\mathbb{R}}$ becomes a $C^\infty$ submanifold of $\mathbb{R}^d$ by the above condition (i).
(We have $\dim_{\mathbb{R}} \underline{X}^{\mathbb{R}} = d-r$.)
We fix a diffeomorphism between $M$ and $\underline{X}^{\mathbb{R}}$ and identify them.
\end{proposition}
Let $\affine^d := \{[1:x_1:\cdots:x_d]\in \mathbb{P}^d(\affine)|\, x_1, \dots,x_d\in \affine\}$, and set 
$\underline{X} := X\cap \affine^d$. $\underline{X}$ is a complex submanifold of $\affine^d$ 
by the condition (i) in Proposition \ref{prop: algebraic model of our manifold}.
\begin{example}
If $M=S^1$, then $R=1$ and $F_1(X_0,X_1,X_2) = -X_0^2+X_1^2+X_2^2$ satisfy the conditions of Proposition 
\ref{prop: algebraic model of our manifold}.
In this case, we have $d=2$ and $r=1$.
\end{example}
\begin{lemma}  \label{lemma: existence of good function, first step}
For any positive integer $N$, there is a homogeneous polynomial 
$\Psi(X_0, X_1, \dots, X_d)\in \affine[X_0,X_1,\dots,X_d]$ of degree $N$ satisfying the 
following two conditions.

\noindent 
(i) 
\begin{equation*}
  \mathrm{rank}
  \begin{bmatrix}
  \frac{\partial\Psi}{\partial X_1} & \cdots & \frac{\partial \Psi}{\partial X_d} \\
  \frac{\partial F_1}{\partial X_1} & \cdots & \frac{\partial F_1}{\partial X_d} \\
  \vdots & \ddots & \vdots \\
  \frac{\partial F_R}{\partial X_1} & \cdots & \frac{\partial F_R}{\partial X_d}
  \end{bmatrix}
  = r + 1 \quad \text{on $X_\infty = X\cap \{X_0=0\}$}.
\end{equation*}

\noindent 
(ii) 
The holomorphic function $\psi:\underline{X}\to \affine$, $[1:x_1:\cdots:x_d]\mapsto \Psi(1,x_1,\dots,x_d)$, is a 
Morse function, i.e., the Hessians of $\psi$ at the critical points are regular.
\end{lemma}
\begin{proof}
From the conditions (i) and (ii) in Proposition \ref{prop: algebraic model of our manifold}, 
$X_\infty$ is a complex submanifold of $\{X_0=0\} = \mathbb{P}^{d-1}(\affine)$ of codimension $r$.
For any $N\geq 1$, we can choose a homogeneous polynomial $\Psi_0(X_1,\dots,X_d) \in \affine[X_1,\dots,X_d]$
of degree $N$ such that the hypersurface $\{\Psi_0 =0\}$ is non-singular ($\mathrm{grad}\Psi_0 \neq 0$ on 
$\{\Psi_0=0\}$) and
transversally intersects with $X_\infty$ in $\mathbb{P}^{d-1}(\affine)$
(the $N$-times Segre embeddings and Bertini's theorem). This implies
\begin{equation*}
  \rank  \begin{bmatrix}
         \frac{\partial\Psi_0}{\partial X_1} & \cdots & \frac{\partial \Psi_0}{\partial X_d} \\
         \frac{\partial F_1}{\partial X_1} & \cdots & \frac{\partial F_1}{\partial X_d} \\
         \vdots & \ddots & \vdots \\
         \frac{\partial F_R}{\partial X_1} & \cdots & \frac{\partial F_R}{\partial X_d}
         \end{bmatrix}
          = r + 1 \quad \text{on $X_\infty$}.
\end{equation*}
For $\mathbf{a} = (a_1, \dots,a_d)\in \affine^d$, we set 
$\Psi_{\mathbf{a}}(X_0,X_1,\dots,X_d) := \Psi_0(X_1\dots,X_d) + X_0^{N-1}(a_1X_1+\cdots+a_dX_d)$.
The above condition (i) (of this lemma) is an open condition.
Hence if we choose $\mathbf{a}\in \affine^d$ sufficiently small in the Euclidean norm, 
then $\Psi_{\mathbf{a}}$ also satisfies the condition (i).
Set $\psi_{\mathbf{a}}([1:x_1: \cdots:x_d]) := \Psi_{\mathbf{a}}(1,x_1,\dots,x_d) 
   = \Psi_0(1,x_1,\dots,x_d) + a_1 x_1 + \cdots + a_d x_d$.
Let $\mathcal{X} \subset \underline{X}\times \affine^d$ be the set of points 
$([1:x_1:\cdots:x_d], \mathbf{a})\in \underline{X}\times \affine^d$ such that
$d(\psi_{\mathbf{a}}|_{\underline{X}})=0$ at $[1:x_1:\cdots:x_d]$.
$\mathcal{X}$ is a complex submanifold of $\underline{X}\times \affine^d$.
By Sard's theorem, we can choose a sufficiently small $\mathbf{a}\in \affine^d$ such that 
$\mathbf{a}$ is a regular value of the projection $\mathcal{X}\to \affine^d$.
Thus we can choose a sufficiently small $\mathbf{a}\in \mathbb{C}^d$ such that 
$\Psi_{\mathbf{a}}$ satisfies the conditions (i) and (ii) of this lemma.
\end{proof}

\subsection{Proof of Theorem \ref{thm: main theorem}}
Let $N$ be a positive integer, and let $V_N\subset \affine[x_1, \dots, x_d, y_1, \dots, y_d]$ be the 
set of polynomials $\varphi(x_1, \dots, x_d, y_1, \cdots, y_d)$ satisfying $\deg \varphi\leq N$.
Take $\varphi\in V_N$. 
We define a homogeneous polynomial $\Phi(Z, X_1,\dots,X_d,Y_1,\dots,Y_d)\in \affine[Z,X_1,\dots,X_d,Y_1,\dots,Y_d]$ by setting 
\begin{equation} \label{eq: definition of Phi from varphi}
   \Phi(Z, X_1,\dots,X_d,Y_1,\dots,Y_d) := 
   Z^N\varphi\left(\frac{X_1}{Z},\dots,\frac{X_d}{Z},\frac{Y_1}{Z},\dots,\frac{Y_d}{Z}\right).
\end{equation}
For positive integers $n$, we set 
\begin{equation} \label{eq: definition of varphi_n and Phi_n}
   \varphi_n(\mathbf{x}_1, \mathbf{x}_2,\dots,\mathbf{x}_{n+1}) := \sum_{k=1}^n \varphi(\mathbf{x}_k,\mathbf{x}_{k+1}),
   \quad \Phi_n(Z, \mathbf{X}_1,\mathbf{X}_2, \dots, \mathbf{X}_{n+1}) := 
   \sum_{k=1}^n \Phi(Z,\mathbf{X}_k,\mathbf{X}_{k+1})
\end{equation}
where $\mathbf{x}_k = (x_{k1},x_{k2},\dots,x_{kd})$ and $\mathbf{X}_k = (X_{k1},X_{k2},\dots,X_{kd})$. 

We define $\rho_{n,k}(\varphi)(Z, \mathbf{X}_1, \dots,\mathbf{X}_{n+1})$ $(1\leq k\leq n+1)$ as the rank of the following matrix:
\begin{equation} \label{eq: definition of rho_n(varphi)}
   \begin{bmatrix}
   \frac{\partial \Phi_n}{\partial X_{k1}}(Z, \mathbf{X}_1,\dots,\mathbf{X}_{n+1}) &
   \frac{\partial \Phi_n}{\partial X_{k2}}(Z, \mathbf{X}_1,\dots,\mathbf{X}_{n+1})  &\cdots & 
   \frac{\partial \Phi_n}{\partial X_{kd}}(Z, \mathbf{X}_1,\dots,\mathbf{X}_{n+1}) \\
   \frac{\partial F_1}{\partial X_1}(Z, \mathbf{X}_k) & \frac{\partial F_1}{\partial X_2}(Z,\mathbf{X}_k)
    & \cdots & \frac{\partial F_1}{\partial X_d}(Z, \mathbf{X}_k) \\
   \vdots & \vdots & \ddots & \vdots \\
   \frac{\partial F_R}{\partial X_1}(Z, \mathbf{X}_k) &\frac{\partial F_R}{\partial X_2}(Z, \mathbf{X}_k) 
    & \cdots & \frac{\partial F_R}{\partial X_d}(Z, \mathbf{X}_k)
   \end{bmatrix}.
\end{equation} 
Consider the following condition for 
$[Z:\mathbf{X}_1:\mathbf{X}_2:\cdots:\mathbf{X}_{n+1}]\in \mathbb{P}^{d(n+1)}(\affine)$:
\begin{equation} \label{eq: critical point set, projectivized}
  F_i(Z,\mathbf{X}_k) =0, \> \rho_{n,k}(\varphi)(Z, \mathbf{X}_1,\dots,\mathbf{X}_{n+1})\leq r\quad 
  (1\leq i\leq R, 1\leq k\leq n+1).
\end{equation}
A point $([1:\mathbf{x}_1], [1:\mathbf{x}_2], \dots, [1:\mathbf{x}_{n+1}])\in \underline{X}^{n+1}$
is a critical point of the function 
\begin{equation}  \label{eq: definition of the restriction of varphi_n}
 \varphi_n|_{\underline{X}^{n+1}}: \underline{X}^{n+1}\to \affine, \quad  ([1:\mathbf{x}_1], \dots, [1:\mathbf{x}_{n+1}])
   \mapsto \varphi_n(\mathbf{x}_1,\dots,\mathbf{x}_{n+1})
\end{equation}
if and only if the point $[1:\mathbf{x}_1:\mathbf{x}_2:\cdots:\mathbf{x}_{n+1}] \in \mathbb{P}^{d(n+1)}(\affine)$ 
satisfies the above condition (\ref{eq: critical point set, projectivized}).
\begin{lemma} \label{lemma: existence of good function, iterated}
There exists $\varphi\in V_N$ satisfying the following (i) and (ii):

\noindent 
(i) For any positive integer $n$,
if a point $[Z:\mathbf{X}_1:\mathbf{X}_2:\cdots:\mathbf{X}_{n+1}]\in \mathbb{P}^{d(n+1)}(\affine)$
satisfies (\ref{eq: critical point set, projectivized}), then $Z\neq 0$.

\noindent 
(ii) For all positive integers $n$, the functions $\varphi_n|_{\underline{X}^{n+1}}:\underline{X}^{n+1}\to \affine$
in (\ref{eq: definition of the restriction of varphi_n}) are Morse functions.
\end{lemma}
\begin{proof}
Let $\Psi(X_0,X_1,\dots,X_d)$ be the homogeneous polynomial of degree $N$ given by Lemma 
\ref{lemma: existence of good function, first step}.
We set $\varphi(\mathbf{x},\mathbf{y}) := \Psi(1,\mathbf{x}) + \Psi(1,\mathbf{y})$
($\mathbf{x} = (x_1, \dots, x_d)$ and $\mathbf{y} = (y_1,\dots,y_d)$).
The polynomials $\varphi_n$ and $\Phi_n$ (defined in (\ref{eq: definition of varphi_n and Phi_n})) become
\begin{equation*}
  \begin{split}
  \varphi_n(\mathbf{x}_1, \mathbf{x}_2,\dots,\mathbf{x}_{n+1}) &= 
   \Psi(1,\mathbf{x}_1) + 2\left(\Psi(1,\mathbf{x}_2) +\cdots + \Psi(1,\mathbf{x}_n)\right) + \Psi(1,\mathbf{x}_{n+1}),\\
  \Phi_n(Z,\mathbf{X}_1,\mathbf{X}_2,\dots,\mathbf{X}_{n+1}) &= 
   \Psi(Z,\mathbf{X}_1) + 2\left(\Psi(Z, \mathbf{X}_2) + \cdots + \Psi(Z, \mathbf{X}_n)\right) +\Psi(Z,\mathbf{X}_{n+1}).
  \end{split}
\end{equation*}
Then the above conditions (i) and (ii) immediately follow from the 
conditions (i) and (ii) of Lemma \ref{lemma: existence of good function, first step}.
\end{proof}
For $n\geq 1$, we define $V_{N,n}\subset V_N$ as the set of $\varphi\in V_N$ satisfying the following:
If a point $[Z:\mathbf{X}_1:\mathbf{X}_2:\cdots:\mathbf{X}_{n+1}]\in \mathbb{P}^{d(n+1)}(\affine)$
satisfies (\ref{eq: critical point set, projectivized}), 
then $Z\neq 0$.
\begin{lemma}  \label{lemma: V_{N,n} is Zariski open}
The set $V_{N,n}$ is a non-empty Zariski open subset of $V_N$.
Here we naturally identify $V_N$ with the affine space $\affine^{\binom{N+2d}{2d}}$.
(Here and in the following in this section, we use only algebraic (not analytic) Zariski open/closed subsets.)
\end{lemma}
\begin{proof}
From Lemma \ref{lemma: existence of good function, iterated}, $V_{N,n}\neq \emptyset$.
We want to show that this is Zariski open.
We define $A\subset V_N\times \mathbb{P}^{d(n+1)-1}(\affine)$ by setting 
\[ A:= \{(\varphi, [\mathbf{X}_1:\cdots:\mathbf{X}_{n+1}])\in V_N\times \mathbb{P}^{d(n+1)-1}(\affine)|\, 
        \text{$\varphi$ and $[0:\mathbf{X}_1:\cdots:\mathbf{X}_{n+1}]$ satisfy (\ref{eq: critical point set, projectivized})}\}.\]
$A$ is a Zariski closed subset of $V_N\times \mathbb{P}^{d(n+1)-1}(\affine)$.
Let $\pi:V_N\times \mathbb{P}^{d(n+1)-1}(\affine)\to V_N$ be the natural projection.
(Here we consider $\pi$ as a map in the algebraic category (not in the analytic category).)
Since $\mathbb{P}^{d(n+1)-1}(\affine)$ is complete (see Mumford \cite[p. 55, Theorem 1]{Mumford-red book}),
$\pi(A)$ is Zariski closed in $V_N$.
Therefore $V_{N,n} = V_N\setminus \pi(A)$ is Zariski open.
\end{proof}
For $n\geq 1$, we define $U_{N,n}\subset V_{N,n}$ by 
\[ U_{N,n} := \{\varphi\in V_{N,n}|\, 
   \varphi_n|_{\underline{X}^{n+1}}:\underline{X}^{n+1}\to \affine \text{ is a Morse function}\}.\]
Here $\varphi_n|_{\underline{X}^{n+1}}$ is the function defined by (\ref{eq: definition of varphi_n and Phi_n}) 
and (\ref{eq: definition of the restriction of varphi_n}).
\begin{lemma} \label{U_{N,n} is non-empty Zariski open}
The set $U_{N,n}$ is a non-empty Zariski open subset of $V_N$.
\end{lemma}
\begin{proof}
From Lemma \ref{lemma: existence of good function, iterated}, we have $U_{N,n}\neq \emptyset$.
We define $A\subset V_{N,n}\times \mathbb{P}^{d(n+1)}(\affine)$ as the set of 
$(\varphi, [Z:\mathbf{X}_1:\dots:\mathbf{X}_{n+1}])\in V_{N,n}\times \mathbb{P}^{d(n+1)}(\affine)$
satisfying (\ref{eq: critical point set, projectivized}).
$A$ is a Zariski closed subset of $V_{N,n}\times \mathbb{P}^{d(n+1)}(\affine)$.
Let $i: V_{N,n}\times (\affine^d)^{n+1} \hookrightarrow V_{N,n}\times \mathbb{P}^{d(n+1)}(\affine)$ be the natural 
open immersion defined by 
$(\varphi, ([1:\mathbf{x}_1], \dots, [1:\mathbf{x}_{n+1}]))\mapsto (\varphi, [1:\mathbf{x}_1:\dots:\mathbf{x}_{n+1}])$.
From the definition of $V_{N,n}$, we have $A\subset i(V_{N,n}\times \underline{X}^{n+1})$.

We define $B \subset V_{N,n}\times \underline{X}^{n+1}$ as the set of 
$\left(\varphi, ([1:\mathbf{x}_1],\dots,[1:\mathbf{x}_{n+1}])\right)\in V_{N,n}\times \underline{X}^{n+1}$
such that $([1:\mathbf{x}_1],\dots,[1:\mathbf{x}_{n+1}])$ is a 
degenerate critical point of $\varphi_n|_{\underline{X}^{n+1}}:\underline{X}^{n+1}\to \affine$.
$B$ is a Zariski closed subset of $V_{N,n}\times \underline{X}^{n+1}$.
Since $i(B)\subset A\subset i(V_{N,n}\times \underline{X}^{n+1})$,
$i(B)$ is a Zariski closed subset of $A$.
Hence $i(B)$ is Zariski closed in $V_{N,n}\times \mathbb{P}^{d(n+1)}(\affine)$.

Let $\pi:V_{N,n}\times \mathbb{P}^{d(n+1)}(\affine) \to V_{N,n}$ be the natural projection.
Since $\mathbb{P}^{d(n+1)}(\affine)$ is complete (\cite[p. 55, Theorem 1]{Mumford-red book}),
$\pi(i(B))$ is Zariski closed in $V_{N,n}$.
Therefore $U_{N,n} = V_{N,n}\setminus \pi(i(B))$ is Zariski open.
\end{proof}
We need the following general (and standard) fact.
\begin{lemma} \label{lemma: Zariski open set is open dense}
Let $K$ be a positive integer, and let $U$ be a non-empty Zariski open subset of $\affine^K$.
Then $U\cap \mathbb{R}^K$ is open dense in $\mathbb{R}^K$ with respect to the Euclidean topology.
\end{lemma}
\begin{proof}
$U\cap \mathbb{R}^K$ is obviously open.
Note the following fact:
If $f(x_1, \dots,x_K)\in \affine[x_1,\dots,x_K]$ vanishes over a non-empty open set (of the 
Euclidean topology) in $\mathbb{R}^K$, then $f=0$.
Therefore $\mathbb{R}^K\setminus U$ cannot have an interior point in $\mathbb{R}^K$.
Hence $U\cap \mathbb{R}^K$ is dense in $\mathbb{R}^K$.
\end{proof}
We set $V_N^{\mathbb{R}} := V_N\cap \mathbb{R}[x_1,\dots,x_d,y_1,\dots,y_d]$ and 
$U_{N,n}^{\mathbb{R}} := U_{N,n}\cap \mathbb{R}[x_1,\dots,x_d,y_1,\dots,y_d]$.
We naturally identify $V_{N}^{\mathbb{R}}$ with the Euclidean space $\mathbb{R}^{\binom{N+2d}{2d}}$.
From Lemma \ref{U_{N,n} is non-empty Zariski open} and Lemma \ref{lemma: Zariski open set is open dense}, 
$U_{N,n}^{\mathbb{R}}$ is open dense in $V_N^{\mathbb{R}}$ with respect to the Euclidean topology.
Set $U_N^{\mathbb{R}} := \bigcap_{n\geq 1}U_{N,n}^{\mathbb{R}}$.
$U_N^{\mathbb{R}}$ is residual (and hence dense) in $V_N^{\mathbb{R}}$ with respect to the Euclidean topology.
If $\varphi\in U_N^{\mathbb{R}}$, then for all $n\geq 1$ 
the functions $\varphi_n|_{\underline{X}^{n+1}}:\underline{X}^{n+1}\to \affine$ are Morse functions.
Recall that we have identified $M$ with $\underline{X}^{\mathbb{R}} = \underline{X}\cap \mathbb{R}^d$.
Hence the above implies that for all $n\geq 1$ the functions
$\varphi_n|_{M^{n+1}}:M^{n+1}\to \mathbb{R}$ are Morse functions.

Then we can prove Theorem \ref{thm: main theorem}:
\begin{proof}[Proof of Theorem \ref{thm: main theorem}]
We define a set $\mathcal{D} \subset C^\infty(M\times M)$ by 
\begin{equation}\label{eq: definition of D}
 \mathcal{D} := \bigcup_{N\geq 1} \{\varphi|_{M\times M}|\, \varphi\in U_N^{\mathbb{R}}\}.
\end{equation}
We will shows that $\mathcal{D}$ is dense in $C^\infty(M\times M)$ 
and that it satisfies the conditions (i) and (ii) in Theorem \ref{thm: main theorem}.
The condition (i) immediately follows from the above argument.

Let $f\in C^\infty(M\times M)$, and let $W$ be an open neighborhood of $f$ in $C^\infty(M\times M)$.
There exists a real polynomial $\phi\in \mathbb{R}[x_1,\dots,x_d,y_1,\dots,y_d]$
such that $\phi|_{M\times M}\in W$ (Weierstrass's theorem).
Set $N := \max(1, \deg \phi)$. We have $\phi\in V_N^{\mathbb{R}}$.
Since $U_N^{\mathbb{R}}$ is dense in $V_N^{\mathbb{R}}$, there exists $\varphi\in U_N^{\mathbb{R}}$ such that 
$\varphi|_{M\times M}\in W$.
This shows that $\mathcal{D}$ is dense in $C^\infty(M\times M)$.

Next we want to show that $\mathcal{D}$ satisfies the condition (ii).
Let $\varphi\in U_N^{\mathbb{R}}$, and let $n$ be a positive integer.
Critical points of $\varphi_n|_{M^{n+1}}:M^{n+1}\to \mathbb{R}$ are also critical points of
$\varphi_n|_{\underline{X}^{n+1}}:\underline{X}^{n+1}\to \affine$. Hence 
\begin{equation}  \label{eq: real critical points v.s. projectivized critical points}
  \#\Cr(\varphi_n|_{M^{n+1}}) \leq \#\{[Z:\mathbf{X}_1:\cdots:\mathbf{X}_{n+1}]\in \mathbb{P}^{d(n+1)}(\affine)|
    \, \text{the condition (\ref{eq: critical point set, projectivized})}\}.
\end{equation}
From the definition of $U_N^{\mathbb{R}}$, the condition (\ref{eq: critical point set, projectivized}) 
implies $Z\neq 0$. 
Moreover all critical points of $\varphi_n|_{\underline{X}^{n+1}}:\underline{X}^{n+1}\to \affine$ are
non-degenerate.
Hence the right-hand-side of (\ref{eq: real critical points v.s. projectivized critical points}) is finite.
We can evaluate it by using B\'{e}zout's theorem \cite[p. 148, Example 8.4.7]{Fulton}
(i.e. by counting the degrees of the equations).
The condition $\rho_{n,k}(Z, \mathbf{X}_1,\dots,\mathbf{X}_{n+1})\leq r$ is equivalent to 
the condition that all $(r+1)$-th sub-determinants of the matrix (\ref{eq: definition of rho_n(varphi)}) are zero.
Set $A:=\max(N, \deg F_1, \dots, \deg F_R)$.
By B\'{e}zout's theorem, the right-hand-side of (\ref{eq: real critical points v.s. projectivized critical points})
is bounded by 
\[ \left(A^{R+(r+1)\binom{d}{r+1}\binom{R+1}{r+1}}\right)^{n+1}.\]
(We can directly evaluate $\#\Cr(\varphi_n|_{M^{n+1}})$ by \cite[Theorem 2]{Milnor-real varieties} 
instead of B\'{e}zout's theorem, 
although the argument of \cite{Milnor-real varieties} also uses B\'{e}zout's theorem.)
\end{proof}

\subsection{Proof of Proposition \ref{prop: algebraic model of our manifold}} 
\label{section:  Proof of Proposition (algebraic model)}
Let $M$ be a compact connected $C^\infty$ manifold without boundary.
We will prove Proposition \ref{prop: algebraic model of our manifold} in this subsection.
We need the following proposition.
Professor Masahiro Shiota kindly explained this result to the authors.
Probably Proposition \ref{prop: algebraic model of our manifold, preliminary} and its proof
are well-known to some specialists of real algebraic geometry.
For example, Kucharz \cite[p. 128]{Kucharz} describes the sketch of the proof of
almost the same result.
\begin{proposition} \label{prop: algebraic model of our manifold, preliminary}
There exist homogeneous polynomials $G_i(X_0,X_1,\dots,X_e)$
$(1\leq i\leq S)$ in $\mathbb{R}[X_0,X_1,\dots,X_e]$ satisfying the following conditions.

\noindent 
(i) The scheme $\mathrm{Proj}(\mathbb{R}[X_0,X_1,\dots,X_e]/(G_1, \dots,G_S))$ is an 
integral regular scheme.

\noindent 
(ii) 
The space 
\[ \{[X_0:X_1:\cdots:X_e]\in \mathbb{P}^e(\mathbb{R})|\, G_i(X_0,X_1,\dots,X_e)=0 \, (1\leq i\leq S)\} \]
 is diffeomorphic to $M$. 
\end{proposition}
\begin{proof}
The idea of the proof is the same as \cite[p. 128]{Kucharz}.
From Nash-Tognoli-King's theorem (\cite{King} and \cite[Chapter 14, Remark 14.1.12]{Bochnak-Coste-Roy}),
there exists a nonsingular real algebraic set $V\subset \mathbb{P}^p(\mathbb{R})$
such that $V$ is diffeomorphic to $M$.
(For the meaning of the term ``nonsingular real algebraic set'', see \cite[Section 3.3]{Bochnak-Coste-Roy}.)
Since $V\cong M$ is connected, it is also Zariski connected. 
Hence $V$ is irreducible in Zariski topology.
(Since $V$ is nonsingular, irreducible components of $V$ do not intersect with each other.
Hence every irreducible component of $V$ is Zariski open and Zariski closed.
See \cite[Theorem 2.8.3, Proposition 3.3.10]{Bochnak-Coste-Roy}.)

Let $I\subset \mathbb{R}[X_0,X_1,\dots,X_p]$ be the homogeneous ideal generated by 
homogeneous polynomials $f\in \mathbb{R}[X_0,X_1,\dots,X_p]$ vanishing on $V$.
Since $V$ is irreducible, $I$ is a prime ideal.
Set $X := \mathrm{Proj}(\mathbb{R}[X_0,X_1,\dots,X_p]/I)$.
$X$ is an integral scheme over $\mathbb{R}$.
From \cite[p. 132, Main theorem I]{Hironaka}, there is a closed subscheme $D$ of $X$ satisfying the following two 
conditions (a) and (b):

\noindent 
(a) The set of points of $D$ is equal to the set of singular points of $X$.

\noindent 
(b) If $m:\Tilde{X}\to X$ is the monoidal transformation of $X$ with center $D$, then 
$\Tilde{X}$ is an integral regular scheme over $\mathbb{R}$.

$m|_{\tilde{X}\setminus m^{-1}(D)}: \tilde{X}\setminus m^{-1}(D) \to X\setminus D$ is 
(algebraically) isomorphic.
For generalities on monoidal transformation (or blowing-up), see Hironaka \cite[pp. 123-130]{Hironaka} and
Hartshorne \cite[pp. 160-169]{Hartshorne}.
Let $X(\mathbb{R})$ (resp. $\Tilde{X}(\mathbb{R})$) be the set of $\mathbb{R}$-morphisms $\mathrm{Spec}\,\mathbb{R} \to X$
(resp. $\mathrm{Spec}\,\mathbb{R}\to \Tilde{X}$).
The images of all $\mathbb{R}$-morphisms $\mathrm{Spec}\,\mathbb{R}\to X$ are regular points of $X$.
Hence $X(\mathbb{R}) \cap D=\emptyset$.
Therefore the natural map $\Tilde{X}(\mathbb{R}) \to X(\mathbb{R})$ is a diffeomorphism.
($X(\mathbb{R})$ and $\Tilde{X}(\mathbb{R})$ naturally become $C^\infty$ manifolds.)
In particular they are both diffeomorphic to $V\cong M$.

Since $X$ is projective over $\mathbb{R}$, $\Tilde{X}$ is also projective over $\mathbb{R}$.
Hence there is a homogeneous ideal $J \subset \mathbb{R}[X_0,X_1,\dots,X_e]$
such that $\Tilde{X}$ is isomorphic to $\mathrm{Proj}(\mathbb{R}[X_0,X_1,\dots,X_e]/J)$ over $\mathbb{R}$.
Let $G_1, \dots,G_S$ be homogeneous polynomials generating $J$.
Then these polynomials satisfy the above conditions (i) and (ii).
\end{proof}
Set $d:=(e+1)^2-1$.
Let $\mathbb{R}[X_{00},\dots,X_{ee}]$ be the polynomial ring of the $d+1$ variables $X_{ij}$ $(0\leq i,j\leq e)$.
Consider a $\mathbb{R}$-homomorphism from $\mathbb{R}[X_{00},\dots,X_{ee}]$ to $\mathbb{R}[X_0,\dots,X_e]$ defined by 
\[ X_{00}\mapsto \sum_{i=0}^e X_i^2, \quad X_{ij}\mapsto X_iX_j \> ((i,j)\neq (0,0)).\]
Let $f:\mathrm{Proj}(\mathbb{R}[X_0,\dots,X_e])\to \mathrm{Proj}(\mathbb{R}[X_{00},\dots,X_{ee}])$
be the $\mathbb{R}$-morphism defined by the above homomorphism.
The map $f$ is a closed immersion (cf. Segre embedding).
Moreover $f$ satisfies 
\[ f(\mathbb{P}^e(\mathbb{R}))
   \subset \mathbb{R}^d := \{[X_{00}:\cdots:X_{ee}]\in \mathbb{P}^d(\mathbb{R})|\, X_{00}\neq 0\}.\]
(This fact is used in \cite[p. 72]{Bochnak-Coste-Roy} to show that real projective spaces are 
affine varieties.)
From this argument and Proposition \ref{prop: algebraic model of our manifold, preliminary} we get the following:
\begin{corollary} \label{cor: algebraic model of our manifold}
There exist homogeneous polynomials
$F_i(X_0,X_1, \dots, X_d)$ $(1\leq i\leq R)$ in $\mathbb{R}[X_0,X_1,\dots,X_d]$ satisfying the
following conditions (i) and (ii):

\noindent 
(i) $\mathrm{Proj}\left(\mathbb{R}[X_0,X_1,\dots,X_d]/(F_1,\dots,F_R)\right)$ is an integral regular scheme.

\noindent 
(ii) The space 
\[ \{[X_0:X_1:\cdots:X_d]\in \mathbb{P}^d(\mathbb{R})|\, F_i(X_0,X_1,\dots,X_d)=0 \, (1\leq i\leq R)\} \]
is diffeomorphic to $M$.
Moreover, if a point $[X_0:X_1:\cdots:X_d]\in \mathbb{P}^d(\mathbb{R})$ satisfies $F_i(X_0,\dots,X_d)=0$ $(1\leq i\leq R)$,
then $X_0\neq 0$.
\end{corollary}
\begin{proof}[Proof of Proposition \ref{prop: algebraic model of our manifold}]
Let $F_i(X_0,X_1, \dots, X_d)$ $(1\leq i\leq R)$ be the polynomials introduced in Corollary 
\ref{cor: algebraic model of our manifold}.
The condition (i) of Corollary \ref{cor: algebraic model of our manifold} implies that the scheme
\[ \mathrm{Proj}\left(\affine[X_0,\dots,X_d]/(F_1,\dots,F_R)\right) 
= \mathrm{Proj}\left(\mathbb{R}[X_0,\dots,X_d]/(F_1,\dots,F_R)\right)\times_{\mathbb{R}}\affine \]
is an equidimensional regular scheme.

Let $X$ be the complex submanifold of $\mathbb{P}^d(\affine)$ which $F_i$ $(1\leq i\leq R)$ define.
From Bertini's theorem (Hartshorne \cite[p. 179, Theorem 8.18]{Hartshorne}),
there exists a non-empty Zariski open set $U\subset \mathbb{P}^d(\affine)$ such that 
for any $[a_0:a_1:\cdots:a_d]\in U$ the hyperplane $a_0X_0+a_1X_1+\cdots+a_dX_d=0$ transversally 
intersects with $X$ in $\mathbb{P}^d(\affine)$.
$U\cap \mathbb{P}^d(\mathbb{R})$ is open dense in $\mathbb{P}^d(\mathbb{R})$ with respect to the 
Euclidean topology. (See Lemma \ref{lemma: Zariski open set is open dense}.)

We have $X\cap \mathbb{R}^d = X\cap \mathbb{P}^d(\mathbb{R})$. Hence $X\cap \mathbb{R}^d$ is compact.
Therefore there exists $[a_0:a_1:\cdots:a_d]\in U\cap \mathbb{P}^d(\mathbb{R})$ such that 
the hyperplane $a_0X_0+a_1X_1+\cdots+a_dX_d=0$ does not intersect with $X\cap \mathbb{R}^d$.
Then, by using a real projective transformation which transforms $a_0X_0+a_1X_1+\cdots+a_dX_d=0$ to $X_0=0$,
we can adjust the polynomials $F_i$ 
so that they satisfy the conditions (i), (ii), (iii) in Proposition \ref{prop: algebraic model of our manifold}.
\end{proof}

\section{Proof of Theorem \ref{thm:growth}}

In this section, for a $C^\infty$ manifold $M$ (not necessarily compact), 
the space $C^\infty(M)$ of real valued $C^\infty$ functions on $M$ 
is endowed with the $C^\infty$ compact-open topology.

\subsection{Interpretation to a dynamical problem}
\label{sec:generating}
For a $C^\infty$ function $H$ on $\RR^2$,
 we denote the partial derivative of $H$ with respected to
 the first and the second coordinates by $\del_1 H$ and $\del_2 H$
 respectively.
By $p_1$ and $p_2$, we denote the projection from $\RR^2$
 to the first and the second coordinate, respectively.
We say that a map $f$ from $\RR^2$  to a set $S$ is
 {\it $(2\pi\ZZ)^2$-periodic}
 if $f(p+m)=f(p)$ for any $p \in \RR^2$ and $m \in (2\pi\ZZ)^2$.
Let $\cH$ be the space of $C^\infty$ functions $H$ on $\RR^2$
 such that $\del_1\del_2H>0$, and
 both $\del_1 H-p_2$ and $\del_2 H-p_1$ are $(2\pi\ZZ)^2$-periodic.
We denote the identity map of $\RR^2$ by $\Id$ and
 define a diffeomorphism $\Theta$ of $\RR^2$ by $\Theta(x,y)=(y,-x)$.
Let $\cD$ be the set of $C^\infty$ area-preserving diffeomorphisms $F$
 of $\RR^2$ such that
 $F-\Theta$ is $(2\pi\ZZ)^2$-periodic and {\it the twist condition}
\begin{equation}
\label{eqn:twist}
\del_2 (p_1 \circ F)>0
\end{equation}
 holds.
We endow the $C^\infty$ compact-open topology to $\cD$.

The following is a classical result known as the correspondence
 between twist maps and their generating functions.
\begin{proposition}
\label{prop:generating} 
There exists a continuous map $\Phi:\cH \ra \cD$ 
 which satisfies the following properties.

\noindent
(i)  $(x',y')=\Phi(H)(x,y)$ if and only if 
 $(y,y')=(\del_1 H(x,x'), -\del_2 H(x,x'))$
 for any $H \in \cH$ and $(x,y,x',y') \in \RR^4$.

\noindent 
(ii) The map $(\Phi,\ev_0):\cH \ra \cD \times \RR$
 is a homeomorphism, where
 the map $\ev_0:\cH \ra \RR$ is given by $\ev_0(H)=H(0,0)$.
\end{proposition}
\begin{proof}
For $H \in \cH$, we define two maps $\phi_H,\psi_H:\RR^2 \ra \RR^2$ by
\begin{align*}
 \phi_H(x,x') & = \left(x,\del_1 H(x,x')\right)\\
 \psi_H(x,x') & = \left(x',-\del_2 H(x,x')\right).
\end{align*}
Put $g_x(x')=\del_1 H(x,x')$.
Since $dg_x/dx'=\del_2 \del_1 H>0$
 and $g_x(x'+m)=g_x(x')+m$ for any $x' \in \RR$ and $m \in 2\pi\ZZ$,
 the map $g_x(x')=\del_1H(x,x')$ is a diffeomorphism of $\RR$
 for any $x \in \RR$.
Hence, $\phi_H$ is a diffeomorphism of $\RR^2$.
Similarly, so is $\psi_H$.
We define a diffeomorphism $\Phi(H)$ of $\RR^2$ by
 $\Phi(H)=\psi_H \circ \phi_H^{-1}$.
Since the maps $\phi_H - \Id$ and $\psi_H -\Theta$ are $(2\pi\ZZ)^2$-periodic,
 $\Phi(H)-\Theta$ is also $(2\pi\ZZ)^2$-periodic.
By direct computation, we can see that
 $\Phi(H)$ is area-preserving and satisfies the twist condition
 (\ref{eqn:twist}).
Therefore, $\Phi(H)$ is an element of $\cD$.

For $F \in \cD$, we define two maps
 $\hat{\phi}_F, \hat{\psi}_F: \RR^2 \ra \RR^2$ by
\begin{align*}
\hat{\phi}_F(x,y) & = (x,p_1 \circ F(x,y))\\
\hat{\psi}_F(x',y') & = (p_1 \circ F^{-1}(x',y'),x').
\end{align*}
By the twist condition (\ref{eqn:twist}), $\hat{\phi}_F$ is a diffeomorphism.
For $(x',y')=F(x,y)$,
 we have $\hat{\phi}_F(x,y)=\hat{\psi}_F \circ F(x,y)=(x,x')$.
Hence, $\hat{\psi}_F$ is a diffeomorphism and
 $F=\hat{\psi}_F^{-1} \circ \hat{\phi}_F$.
Take $g_1,g_2 \in C^\infty(\RR^2)$ such that
\begin{align*}
 \hat{\phi}_F^{-1}(x,x') & = (x, g_1(x,x'))\\
 \hat{\psi}_F^{-1}(x,x') & = (x',-g_2(x,x')).
\end{align*}
By a direct calculation, we obtain that
\begin{gather*}
\del_1 g_2-\del_2 g_1
 =\log \det DF_{\hat{\phi}_F^{-1}(x,x')}=0,\\
 (\del_2 g_1)^{-1}
 =\del_2(p_1 \circ F)(\hat{\phi}_F^{-1}(x,x'))>0.
\end{gather*}
The former implies that the one-form
 $g_1 dx+ g_2 dx'$ is closed.
By Poincar\'e's lemma,
 there exists a unique $C^\infty$-function $H_F$ such that
 $H_F(0,0)=0$ and $dH_F = g_1 dx + g_2 dx'$.
Since $F-\Theta$ is $(2\pi\ZZ)^2$-periodic,
 the maps $\hat{\phi}_F-\Id$ and $\hat{\psi}_F-\Theta^{-1}$
 are $(2\pi\ZZ)^2$-periodic.
This implies that $\hat{\phi}_F^{-1}-\mathrm{Id}$ and $\hat{\psi}_F^{-1}-\Theta$
 are $(2\pi\ZZ)^2$-periodic, and hence, so are $g_1-p_2$ and $g_2-p_1$.
Therefore, $H_F$ is an function in $\cH$.
We define a map $\Psi:\cD \times \RR \ra \cH$ by $\Psi(F,c)=H_F+c$.
Since
\begin{equation*}
F=\hat{\psi}_F^{-1} \circ \hat{\phi}_F=\psi_{H_F} \circ \phi_{H_F}^{-1} 
 =\Phi(H_F),
\end{equation*}
 $\Psi$ is the inverse of $(\Phi, ev_0)$.
The continuity of $\Phi$ and $\Psi$ follows from the constrictions.
\end{proof}

\begin{corollary}
\label{cor:Crit} 
For $H \in \cH$ and $(x_0,\dots,x_n) \in \RR^{n+1}$,
 the following two conditions are equivalent.

\noindent 
(i)  $(x_0,\dots,x_n)$ is a critical point of $H_n:\mathbb{R}^{n+1}\to \mathbb{R}$ where 
$H_n(x_0, \dots,x_n) := \sum_{i=0}^{n-1} H(x_i,x_{i+1})$.

\noindent 
(ii) There exists $(y_0,\dots,y_n) \in \RR^{n+1}$
 such that $y_0=y_n=0$ and
 $(x_{j+1},y_{j+1})=\Phi(H)(x_j,y_j)$ for any $j=0,\dots,n-1$.
\end{corollary}
\begin{proof}
For $H \in C^\infty(\RR^2)$, 
 a point $(x_0,\dots,x_n) \in \RR^{n+1}$ is a critical point of $H_n$
 if and only if 
 $\del_1 H(x_0,x_1)=\del_2 H(x_{n-1},x_n)=0$
 and
 $\del_1 H(x_j,x_{j+1}) = - \del_2 H(x_{j-1},x_j)$
 for any $j=1,\dots,n-1$.
Hence, the corollary follows from Proposition \ref{prop:generating}.
\end{proof}

Hence, the counting of critical point of $H_n$ is reduced to
 the counting of points in
 $\Phi(H)^{-n}(\RR \times \{0\}) \cap (\RR \times \{0\})$.

\subsection{Abundance of recurrence of intervals}
\label{sec:growth}
Let $\TT^2=(\RR/2\pi\ZZ)^2$ be the two-dimensional torus.
Set $M=\RR^2$ or $\TT^2$.
By $\Diff_\omega(M)$, we denote the set
 of area-preserving diffeomorphisms of $M$
 endowed with the $C^\infty$ compact-open topology.
By $\Int I$, we denote the interior of an interval $I$.
For embedded intervals $I$ and $J$ in $M$,
 let $I \pitchfork J$ be the set of
 transverse intersections of $\Int I$ and $\Int J$.

Let us recall some definitions and known facts on dynamical systems.
The authors recommend \cite{KH} or \cite{Ro} for reference.
For $f \in \Diff_\omega(M)$,
 a fixed point $p$ of $f^N$ is called {\it hyperbolic}
 if no eigenvalues of $Df^N_p$ is of absolute value one.
Remark that one of the eigenvalues is of absolute value greater than one
 and the other is less than one since $f$ is area-preserving.
{\it A continuation} of a hyperbolic fixed point $p$ of $f^N$
 is a continuous map $\hat{p}$
 from a neighborhood $\cU \subset \Diff_\omega(M)$ of $f$ to $M$
 such that $\hat{p}(f)=p$ and $\hat{p}(g)$ is a hyperbolic fixed point
 of $g^N$ for any $g \in \cU$.
It is known that any hyperbolic fixed point admits a continuation.

Let $d$ be the standard distance on $M=\RR^2$ or $\TT^2$. 
For a hyperbolic fixed point $p$ of $f^N$, 
 {\it the stable manifold} $W^s(p;f)$ and
 {\it the unstable manifold} $W^u(p;f)$ are defined by
\begin{align*}
 W^s(p;f)
  & =\left\{q \in M \mid d(f^n(p),f^n(q)) \ra 0\; (n \ra +\infty) \right\},\\
 W^u(p;f)
  & =\left\{q \in M \mid d(f^n(p),f^n(q)) \ra 0\; (n \ra -\infty) \right\}.
\end{align*}
By the stable manifold theorem,
 both $W^s(p;f)$ and $W^u(p;f)$ are $C^\infty$ injectively immersed curves.
For $L>0$, let $W^s_L(p;f)$ and $W^u_L(p;f)$ be
 the compact subintervals of $W^s(p;f)$ and $W^u(p;f)$
 centered at $p$ whose length is $2L$.
The set $W^s_L(p;f)$ satisfies
 $f^N(W^s_L(p;f)) \subset W^s_L(p;f)$
 and $W^s(p;f)=\bigcup_{k \geq 0}f^{-kN}W^s_L(p;f)$.
The set $W^u_L(p;f)$ also have similar properties.
For a continuation $\hat{p}:\cU \ra M$ of $p$,
 it is known that $W^s_L(\hat{p}(f);f)$ and $W^u_L(\hat{p}(f);f)$
 depends continuously on $f$ as $C^\infty$ embedded intervals.

We say that a hyperbolic fixed point $p$ of $f^N$
 {\it exhibits homoclinic tangency}
 at a point $q \in W^s(p;f) \cap W^u(p;f)$
 if $W^s(p;f)$ and $W^u(p;f)$ are tangent at $q$.
We also say that the homoclinic tangency at $q$ is $\infty$-flat if
 the $\infty$-jets of $W^s(p;f)$ and $W^u(p;f)$ at $q$ coincide.

The aim of this section is to show the following.
\begin{proposition}
\label{prop:abundance}
Let $J$ be an interval in $\TT^2$,
 $f_0$ a diffeomorphism in $\Diff_\omega(\TT^2)$,
 and $p_0$ a hyperbolic fixed point of $f_0^N$ for some $N \geq 1$.
Suppose that $p_0$ exhibits homoclinic tangency
 and $W^\sigma(p_0;f_0) \pitchfork J \neq \emptyset$ for each $\sigma=s,u$.
Then, there exists an open subset $\cU_*$ of $\Diff_\omega(\TT^2)$
 such that $f_0 \in \cl{\cU_*}$ and the set
\begin{equation*}
\cU_n=\bigcup_{m \geq n} \left\{f \in \cU_*
 \mid \#\left[f^m(J) \pitchfork J \right] \geq a_m\right\}
\end{equation*}
 is an open dense subset of $\cU_*$
 for any given sequence $(a_m)_{m \geq 1}$ of positive integers
 and any $n \geq 1$.
\end{proposition}

The first ingredient of the proof is
 {\it The Inclination Lemma} (or {\it The $\lambda$-lemma}).
See {\it e.g.} {\cite[Theorem V.11.1]{Ro}} for the proof.
\begin{theorem}
\label{thm:lambda}
Let $p$ be a hyperbolic fixed point of $f \in \Diff_\omega(\TT^2)$,
 $I$ and $J$ embedded compact intervals in $\TT^2$,
 and $K$ a closed subset of $\TT^2$ such that
 $p \in I \subset W^u(p;f)$,
 $J \pitchfork W^s(p;f)$ contains a point $z$,
 and $K \cap [I \cup \{f^n(z); n \geq 0\}]= \emptyset$.
Then, there exists a sequence $(J_k)_{k\geq 1}$ of subintervals of $J$
 such that

 \noindent (i) $f^n(J_k) \cap K= \emptyset$ for any $n=0,\dots,k$.

 \noindent (ii) $f^{k}(J_k)$ converges to $I$
 as a $C^\infty$ embedded interval as $k \ra \infty$.

\end{theorem}
Take a neighborhood $\cU_0 \subset \Diff_\omega(\TT^2)$ of $f_0$
 and a continuation $\hat{p}:\cU_0 \ra \TT^2$ of $p_0$.
Applying The Inclination Lemma, we give a criterion to
 approximation by a diffeomorphism $g$ such that
 $g^m(J) \pitchfork J$ contains infinitely many points.
\begin{lemma}
\label{lemma:intersection}
Let $J^-$ and $J^+$ be compact intervals in $\TT^2$,
 $f$ a diffeomorphism in $\Diff_\omega(\TT^2)$, and
 $p$ a hyperbolic fixed point of $f^N$ with some $N \geq 1$.
Suppose that there exist $L>0$
 and $q \in [W^s(p;f) \setminus W^s_L(p;f)]
 \cap [W^u(p;f) \setminus W^u_L(p;f)]$ such that
 $J^- \pitchfork W^s_L(p;f) \neq \emptyset$,
 $J^+ \pitchfork W^u_L(p;f) \neq \emptyset$, and
 $p$ exhibits $\infty$-flat homoclinic tangency at $q$.
Then, for any give neighborhood $\cU$ of $f$ in $\Diff_\omega(\TT^2)$
 and any $n_0 \geq 1$,
 there exists a diffeomorphism $g \in \cU$ and $n_* \geq n_0$ such that
\begin{equation*}
 \#\left[g^{n_*}(J^-) \pitchfork J^+\right]=\infty.
\end{equation*}
\end{lemma}
\begin{proof}
We may assume that $N$ is the minimal period of $p$,
 {\it i.e.}, the minimal positive integer satisfying $f^N(p)=p$.
Let $m_-$ and $m_+$ be the minimal positive integers
 such that  $f^{-m_-}(q) \in W^u_L(p;f)$ and $f^{m_+}(q) \in W^s_L(p;f)$.
Since $q \in W^s(p;f) \cap W^u(p;f)$ and
 $W^\sigma(p;f) \cap W^\sigma(f^j(p);f)=\emptyset$
 for $\sigma=s,u$ if $j \not\equiv 0 \pmod N$,
 we have $m_\pm \geq N$.
Take $L'>L$ such that
 $f^j(q) \not\in W^u_{L'}(p;f) \cup W^s_{L'}(p;f)$
 for any $j=-m_- +1,\dots,m_+ -1$.
We also take an open neighborhood $U$ of $q$ such that
 $\cl{f^i(U)} \cap [W^u_{L'}(p;f) \cup W^s_{L'}(p;f)]= \emptyset$
 and $f^i(U) \cap f^j(U) =\emptyset$
 for any $i,j=-m_-+1,\dots,m_+-1$ with $i \neq j$.

By The Inclination Lemma,
 there exists a sequence $(J^-_k)_{k \geq 1}$ of subintervals
 of $J^-$ such that $f^{iN}(J^-_k) \cap \bigcup_{j=-m_-+1}^{m_+-1}f^j(U)
=\emptyset$
 for any $i=0,\dots,k$
 and $f^{kN}(J^-_k)$ converges to $W^u_{L'}(p;f)$ as $k \ra \infty$.
Since $m_\pm \geq N$, the former implies that
 $f^j(J^-_k) \cap U=\emptyset$ for any $j=0,\dots,kN+m_--1$.
Take an interval $I^u$ in $f^{m_-}(W^u_{L'}(p;f)) \cap U$
 such that $\Int I^u$ contains $q$.
Then, there exists a sequence $(I^-_k)_{k \geq 1}$ of subintervals of $J^-$
 such that $I^-_k \subset J^-_k$ for any $k \geq 1$
 and $f^{kN+m_-}(I^-_k)$ converges to $I^u$ as $k \ra \infty$.
It satisfies that
 $f^j(I^-_k) \cap U=\emptyset$ for any $j=0,\dots,kN+m_--1$
 and $f^{kN+m_-}(I^-_k)$ converges to $I^u$ as $k \ra \infty$.

Take an interval $I^s$ in $f^{-m_+}(W^s_{L'}(p;f)) \cap U$
 such that $\Int I^s$ contains $q$.
By the same argument as above,
 we can take a sequence $(I^+_k)_{k \geq 1}$ of subintervals $J^+$
 such that $f^{-j}(I^+_k) \cap U=\emptyset$ for any $j=0,\dots, kN+m_+-1$
 and $f^{-(kN+m_+)}(I^+_k)$ converges to $I^s$ as $k \ra \infty$.

Fix an neighborhood $\cU$ of $f$ in $\Diff_\omega(\TT^2)$
 and an integer $n_0 \geq 1$.
Let ${\mathcal V}$ be the set of diffeomorphisms $\phi \in \Diff_\omega(\TT^2)$
 such that $f \circ \phi^{-1} \in \cU$ and $\text{supp}(\phi) \in U$.
Since $I^s$ and $I^u$ are compact intervals in $U$
 and they have the same $\infty$-jets at $q$,
 there exists $k_* \geq n_0$ and $\phi_1 \in \cV$
 such that the set $\phi_1(f^{-(k_*N+m_+)}(I^+_{k_*}))
 \cap f^{k_*N+m_-}(I^-_{k_*})$
 contains an interval.
Put $n_+=k_*N+m_+$ and $n_-=k_*N+m_-$.
Take a small perturbation $\phi_2 \in \cV$ of $\phi_1$ such that
 $\#\left[\phi_2(f^{-n_+}(I^+_{k_*}))
 \pitchfork f^{n_-}(I^-_{k_*})\right] =\infty$.
Put $g=f \circ \phi_2^{-1} \in \cU$.
It is easy to check that
 $g^{n_-}(I^-_{k_*})=f^{n_-}(I^-_{k_*})$
 and $g^{-n_+}(I^+_{k_*})
 =\phi_2 \circ f^{-n_+}(I^+_{k_*})$.
Therefore, we have
$\#\left[g^{n_+ + n_-}(I^-_{k_*})) \pitchfork I^+_{k_*}\right] =\infty$.
Since $k_* \geq n_0$,
 we also have $n_+ + n_- \geq n_0$.
\end{proof}
The other ingredients
 are the following results on homoclinic tangency.
\begin{theorem}
[Duarte \cite{Du}]
\label{thm:Duarte}
Let $f_0$ be a diffeomorphism in $\Diff_\omega(\TT^2)$,
 $p_0$ a hyperbolic fixed point of $f^N$,
 and $\hat{p}:\cU_0 \ra \TT^2$ a continuation of $p_0$
 on an open neighborhood $\cU_0$ of $f_0$.
If $p_0$ exhibits homoclinic tangency,
 then there exists an open set $\cU \subset \cU_0$
 and a dense subset $\cT$ of $\cU$ such that
 $f_0 \in \cl{\cU}$ and $\hat{p}(f)$ exhibits homoclinic tangency
 for any $f \in \cT$.
\end{theorem}
\begin{theorem}
[Gonchenko-Turaev-Shilnikov \cite{GTS}] 
\label{thm:GTS}
Let $f_0$ be a diffeomorphism in $\Diff_\omega(\TT^2)$
 and $p_0$ a hyperbolic fixed point of $f_0^N$.
If $p_0$ exhibits homoclinic tangency,
 then any neighborhood of $f_0$ contains a diffeomorphism
 $g$ such that $p_0$ is a hyperbolic fixed point of $g^N$
 and it exhibits $\infty$-flat homoclinic tangency.
\end{theorem}

Now, we prove Proposition \ref{prop:abundance}.
Let $J$, $f_0$, and $p_0$ be the ones in
 the assumption of the proposition.
Take $L>0$, an open neighborhood $\cU_0$ of $f_0$
 and a continuation $\hat{p}:\cU_0 \ra \TT^2$ of $p_0$ such that
 $W^s_L(\hat{p}(f);f) \pitchfork J \neq \emptyset$
 and $W^u_L(\hat{p}(f);f) \pitchfork J \neq \emptyset$
 for any $f \in \cU_0$.
By the Kupka-Smale Theorem, $W^s(\hat{p}(f);f)$
 and $W^u(\hat{p}(f);f)$ intersect transversely for generic $f \in \cU_0$.
Hence, there exists an open and dense subset $\cU_1$ of $\cU_0$ such that
 $f \in \cl{\cU_1}$ and
 all intersections of $W^s_L(\hat{p}(f);f)$ and $W^u_L(\hat{p}(f);f)$
 are transverse for any $f \in \cU_1$.
For $n \geq 1$, put
\begin{equation*}
 \cT_n=\{f \in \Diff_\omega(\TT^2)
 \mid \#\left[f^n(J) \pitchfork J \right]=\infty\}.
\end{equation*}
By Theorems \ref{thm:Duarte} and \ref{thm:GTS}
 and Lemma \ref{lemma:intersection},
 we can take an open subset $\cU_*$ of $\cU_1$
 such that $f_0 \in \cl{\cU_*}$ and 
 the set $(\bigcup_{m \geq n}\cT_m) \cap \cU_*$ is dense in $\cU_*$
 for any $n \geq 1$.
This implies that the set
$\cU_n 
 = \bigcup_{m \geq n}\{f \in \cU_* \mid
 \#\left[f^m(J) \pitchfork J \right] \geq a_m \}$
 is an open and dense subset of $\cU_*$ for any sequence $(a_m)_{m \geq 1}$
 and any $n \geq 1$.

\subsection{Proof of Theorem \ref{thm:growth}}
In this subsection, we prove Theorem \ref{thm:growth}.
Recall that $\cD$ is the set of diffeomorphisms $F \in \Diff_\omega(\RR^2)$
 such that $F-\Theta$ is $(2\pi\ZZ)^2$-periodic and
 the twist condition (\ref{eqn:twist}) holds,
 where $\Theta(x,y)=(y,-x)$.
Let $\pi_T:\RR^2 \ra \TT^2=(\RR/2\pi\ZZ)^2$ be the natural projection.
It induces a map $\pi_{T*}:\cD \ra \Diff_\omega(\TT^2)$.
For a diffeomorphism $F \in\cD$, an open subset $U$ of $\RR^2$,
 and $n \geq 1$, we put
\begin{equation*}
 \Lambda^{n}(F,U)=\left[(\RR \times \{0\})
 \pitchfork F^{-n}(\RR \times \{0\})\right]
 \cap \bigcap_{m=0}^n F^{-m}(U).
\end{equation*}

First, we ``lift'' Proposition \ref{prop:abundance} to the set $\cD$.
\begin{proposition}
\label{prop:abundance 2}
Suppose that a diffeomorphism $F_0$ in $\cD$,
 a hyperbolic fixed point $p_0$ of $F_0^N$ ($N\geq 1$),
 open subsets $U_0$ and $U_1$ of $\RR^2$,
 and open subset $\cU_0$ of $\cD$ 
 satisfy the following four conditions:

 \noindent 
(i) $p_0 \in U_0 \subset U_1$,
 $\cl{U_1} \subset (-\pi,\pi)^2$, and $F_0 \in \cU_0$.

\noindent (ii) $p_0$ exhibits homoclinic tangency.

\noindent (iii) $W^\sigma(p_0;F_0) \pitchfork
 [(\RR \times \{0\}) \cap U_0] \neq \emptyset$ for $\sigma=s,u$.

 \noindent (iv) $F^n(U_0) \subset U_1$ for any $F \in \cU_0$ and $n \in \ZZ$.

Then, there exists an open subset $\cU$ of $\cU_0$
 such that $F_0 \in \cl{\cU}$ and the set
\begin{equation*}
\bigcup_{m \geq n}
\left\{ F \in \cU \mid \#\Lambda^m(F,U_1) \geq a_m \right\}
\end{equation*}
 is open and  dense in $\cU$
 for any given sequence $(a_m)_{m \geq 1}$ of positive integers
 and $n \geq 1$.
\end{proposition}
\begin{proof}
Put $f_0=\pi_{T*}(F_0)$, $p_T=\pi_T(p_0)$,
 $L_0=(\RR \times \{0\}) \cap U_0$, and $L_T=\pi_T(L_0)$.
The hyperbolic fixed point $p_T$ of $f_0^N$ exhibits homoclinic tangency.
By assumption,
 $W^\sigma(p_T;f_0) \pitchfork L_T \neq \emptyset$ for $\sigma=s,u$.
Applying Proposition \ref{prop:abundance},
 we obtain an open subset $\cU_T$ of $\Diff_\omega(\TT^2)$
 such that $f_0 \in \cl{\cU_T}$ and the set
\begin{equation*}
 \cU_n=\bigcup_{m \geq n}
\left\{f \in \cU_T \mid
 \#\left[L_T \pitchfork f^{-m}(L_T)\right] \geq a_m \right\}
\end{equation*}
 is open and dense in $\cU_T$
 for any sequence $(a_m)_{m \geq 1}$ of positive integers
 and $n \geq 1$.

The set $\pi_{T*}(\cD)$ is an open subset of $\Diff_\omega(\TT^2)$,
 and $\pi_{T*}$ is a covering map
 onto $\pi_{T*}(\cD)$.
Hence, there exists an open subset $\cU$ of $\cU_0$ 
 such that $F_0 \in \cl{\cU}$,
 the restriction of $\pi_{T*}$ to $\cU$
 is a homeomorphism onto a open subset of $\cU_T$.
Since $L_0=\pi_T^{-1}(L_T) \cap U_0$
 and $F^n(U_0) \subset U_1$ for any $F \in \cU$ and $n \in \ZZ$,
 we have $F^n(L_0) \subset U_1$.
This implies that
 $L_0 \pitchfork F^{-n}(L_0) \subset \bigcap_{m=0}^n F^{-m}(U_1)$,
 and hence,
\begin{align*}
 \# \Lambda^n(F,U_1)
 & \geq \#\left[L_0 \pitchfork F^{-n}(L_0)\right]
  = \#\left[L_T \pitchfork (\pi_{T*}(F))^{-n}(L_T) \right].
\end{align*}
Since $\pi_{T*}$ maps $\cU$ to an open subset of $\cU_T$ homeomorphically,
 the set
\begin{equation*}
 \bigcup_{m \geq n}\left\{F \in \cU \mid
 \#\Lambda^m(F,U_1) \geq a_m\right\}
\end{equation*}
 is open and dense in $\cU$
 for any given sequence $(a_m)_{m \geq 1}$ of positive integers
 and any $n \geq 1$.
\end{proof}

Next, we see that the existence of a diffeomorphism $F_0$
 satisfying Proposition \ref{prop:abundance 2} implies Theorem \ref{thm:growth}.
Let $U_0$ and $\cU$ be open subsets of $(-\pi,\pi)^2$ and $\cD$
 in Proposition \ref{prop:abundance 2}.
Put $\cU_H=\Phi^{-1}(\cU)$, where $(\Phi,\ev_0):\cH \ra \cD \times \RR$
 is the homeomorphism given in Proposition \ref{prop:generating}.
Take a compact interval $I \subset (-\pi,\pi)$ such that
 $U_0 \subset I \times \RR$.
By Corollary \ref{cor:Crit},
\begin{equation*}
 \#\left[\Cr(H_n) \cap I^{n+1} \right] \geq \#\Lambda^{n}(\Phi(H),U_0)
\end{equation*}
 for any $H \in \cU_H$.
Fix a sequence $(a_m)_{m \geq 1}$ of positive integers.
Then, the set
\begin{equation*}
\cU_n= \bigcup_{m \geq n}\left\{H \in \cU_H \mid
 \#[\Cr(H_m) \cap I^{m+1}] \geq a_m\right\}
\end{equation*}
 contains an open dense subset of $\cU_H$ for any $n \geq 1$.

Let $C^\infty(I^2)$ be the set of functions on $I^2$
 which extends to a $C^\infty$ function on an open neighborhood of $I^2$.
Recall that we identify $S^1$ with $\RR/2\pi \ZZ$.
Let $r_S:C^\infty(S^1\times S^1) \ra C^\infty(I^2)$
 and $r_H:\cH \ra C^\infty(I^2)$
 be the maps induced by the restriction of functions.
They are continuous and open.
Hence, the set $\cU_S=r_S^{-1}(r_H(\cU_H))$ is open in
 $C^\infty(S^1 \times S^1)$
Moreover, if $\cU'$ is an open dense subset of $\cU_H$,
 then $r_S^{-1}(r_H(\cU'))$ is also open and dense in $\cU_S$.
Since
\begin{equation*}
\#\left[\Cr(f_n) \cap I^{n+1}\right] 
 = \#\left[\Cr(H_n) \cap I^{n+1}\right]
\end{equation*}
 for $f \in \cU_S$ and $H \in \cU_H$ with $r_S(f)=r_H(H)$,
 the set
\begin{equation*}
\cU'_n=\bigcup_{m \geq n} \left\{f \in \cU_S \mid
 \#[\Cr(f_m) \cap I^{m+1}] \geq a_m \right\}
 \supset r_S^{-1}(r_H(\cU_n))
\end{equation*}
 contains an open dense subset of $\cU_S$
 for any given sequence $(a_m)_{m \geq 1}$ of positive integers
 and any $n \geq 1$.
Therefore, the set
\begin{equation*}
 \left\{f \in \cU_S
 \mid \limsup_{n \ra \infty} \frac{\#\Cr(f_n)}{a_n} \geq 1
 \right\}
 \supset \bigcap_{n \geq 1}\cU_n'
\end{equation*}
 is residual in $\cU_S$.
This is just the statement of Theorem \ref{thm:growth}.
 
Finally, we construct a diffeomorphism in $\cD$ which satisfies
 the assumption of Proposition \ref{prop:abundance 2}.
Put $O=(0,0)$ and $p_0=(1,0)$.
For $r>0$, we define a disk $D'(r)$ and an annulus $A(r)$ in $\RR^2$ by
\begin{align*}
D'(r) & =\{q \in \RR^2 \mid \|q-p_0\| \leq r\}\\
A(r) & =\{q \in \RR^2 \mid 2-r \leq \|q\| \leq 2+r\}.
\end{align*}
Here $\|\cdot\|$ is the Euclidean norm.
Fix $\delta=1/6$ and let $G$ be a $(2\pi\ZZ)^2$-periodic
 $C^\infty$ function on $\RR^2$ such that
\begin{equation*}
G(x,y)=
\begin{cases}
(x^2+y^2)^2 & \text{for }(x,y) \in A(\delta)\\
y^2+\left(x-1\right)^2
 \left(x-\left(1+\delta\right)\right)
 & \text{for }(x,y) \in D'(2\delta),
\end{cases}  
\end{equation*}
 and 
\begin{equation*}
 \text{supp}(G) \cap (-\pi,\pi)^2 \subset A(2\delta) \cup D'(3\delta).
\end{equation*}
Let $\Psi_t$ be the flow generated by a vector field
 $X_G=(\del_2 G)\del_1 - (\del_1 G) \del_2$.
It is an area-preserving flow satisfying the following properties
 for any $t>0$:
\begin{enumerate}
\item[(i)] $\text{supp}(\Psi_t) \cap (-\pi,\pi)^2
 \subset A(2\delta) \cup D'(3\delta)$.
\item[(ii)] $\Psi_t(r\cos\theta, r\sin\theta)
 =(r\cos(\theta-4r^2t),r\sin(\theta-4r^2t))$
 for any $(r\cos\theta,r\sin\theta) \in A(\delta)$.
\item[(iii)] $p_0$ is a hyperbolic fixed point of $\Psi_t$.
\item[(iv)] $\{(x,y) \in D'(2\delta) \mid x \geq 1,
 y^2+(x-1)^2(x-(1+\delta))=0\}$ is contained in
 $W^s(p_0;\Psi_t) \cap W^u(p_0;\Psi_t)$.
\end{enumerate}
Remark that the last item implies that
 $p_0$ exhibits homoclinic tangency
 and $W^\sigma(p_0;\Psi_t) \pitchfork (\RR \times \{0\}) \neq \emptyset$
 for $\sigma=s,u$.

Since $\cD$ is an open subset of the set of diffeomorphisms
 $F \in \Diff_\omega(\RR^2)$ such that
 $F-\Theta$ is $(2\pi\ZZ)^2$-periodic,
 there exists small $T>0$ such that $\Theta \circ \Psi_T \in \cD$.
Put $F_0=\Theta \circ \Psi_T$.
Since $F_0^4=\Psi_T$ on $D'(2\delta)$,
 $p_0$ is a hyperbolic fixed point of $F_0^4$
 and it exhibits homoclinic tangency.
For $(r\cos\theta,r\sin\theta) \in A(\delta)$, we have
\begin{equation*}
 F_0^4(r\cos\theta,r\sin\theta)=\Psi_T^4(r\cos\theta, r\sin\theta)
 =(r\cos(\theta-16r^2 T),r \sin(\theta-16r^2 T)).
\end{equation*}
The Kolmogorov-Arnold-Moser Theorem on persistence of
 invariant circle (see {\it e.g.,} \cite{Mo}) implies that
 there exist a neighborhood $\cU_0$ of $F_0$ such that
 any $F \in \cU_0$ admits an $F$-invariant circle $C(F) \subset A(\delta)$ 
 which winds the annulus once.
This implies that open disks $U_0=\{q \in \RR^2 \mid \|q\| < 2-\delta\}$
 and $U_1=\{q \in \RR^2 \mid \|q\| < 2+\delta\}$ satisfy
 that $F^n(U_0) \subset U_1$ for any $F \in \cU_0$ and $n \in \ZZ$.
Therefore, $F_0$ and $p_0$ satisfy the assumption of
 Proposition \ref{prop:abundance 2}.
As we mentioned in the above, it completes the proof of Theorem \ref{thm:growth}.

\section{Proof of Theorem \ref{thm: stable region}} \label{section: proof of stable region}
We denote the natural coordinate of $(S^1)^{n+1} = (\mathbb{R}/2\pi \mathbb{Z})^{n+1}$ 
by $(x_0, x_1,\dots, x_n)$.
Let $F:(S^1)^{n+1}\to \mathbb{R}$ be a smooth function.
For $x\in (S^1)^{n+1}$, we set 
\[\hess_x F := \left(\frac{\partial^2 F(x)}{\partial x_i\partial x_j}\right)_{0\leq i,j\leq n}.\]
This is a $(n+1,n+1)$-matrix. We set 
\[ M(F)(x) := \max_{0\leq i\leq n}\left|\frac{\partial F(x)}{\partial x_i}\right|
   + \inf_{u\in \mathbb{R}^{n+1}, \|u\|=1} \|\hess_x F(u)\|.\]
$F$ is a Morse function if and only if $M(F)(x)>0$ for all $x\in (S^1)^{n+1}$.
(Here $\|\cdot\|$ is the Euclidean norm in $\mathbb{R}^{n+1}$.)

Let $f:S^1\times S^1\to \mathbb{R}$ be a smooth function.
It is easy to check that for any $n\geq 1$
\begin{equation} \label{eq: hessian estimate}
 \sup_{0\leq i\leq n}\left|\frac{\partial f_n(x)}{\partial x_i}\right| \leq 2\norm{\nabla f}_\infty,\quad 
 \|\hess_x (f_n) u\| \leq 5\norm{\nabla^2 f}_\infty \|u\|
\quad (u\in \mathbb{R}^{n+1}),
\end{equation}
where $\norm{\nabla f}_{\infty}$ is the supremum of $|\partial f(x)/\partial x_i|$ over $x\in S^1\times S^1$ and 
$0\leq i\leq 1$, and
$\norm{\nabla^2f}_\infty$ is the supremum of 
$|\partial^2 f (x)/\partial x_i\partial x_j|$ over $x\in S^1\times S^1$ and $0\leq i,j\leq 1$.

Let $h:S^1\to \mathbb{R}$ be a Morse function.
We define a positive number $K$ by
\[ K := \inf_{x\in S^1} M(h)(x) >0.\]
We define 
$g: S^1\times S^1\to \mathbb{R}$ by $g(x,y) := h(x)+h(y)$.
Then 
$g_n(x_0,x_1,\dots,x_n) = h(x_0)+ 2h(x_1)+\dots+2h(x_{n-1})+h(x_n)$.
It is easy to see 
\[ \#\mathrm{Cr}(g_n) = (\#\mathrm{Cr}(h))^{n+1}.\]

The Hessian $\hess_x (g_n)$ is the diagonal matrix $\mathrm{diag}(h''(x_0),2h''(x_1),\dots,2h''(x_{n-1}),h''(x_n))$.
Hence
\[ \inf_{u\in \mathbb{R}^{n+1}, \|u\|=1} 
\|\hess_x g_n(u)\| = 
   \min(|h''(x_0)|, 2|h''(x_1)|,\dots, 2|h''(x_{n-1})|, |h''(x_n)|).\]
\begin{equation*}
  \begin{split}
    M(g_n)(x) = &\max(|h'(x_0)|,2|h'(x_1)|,\dots,2|h'(x_{n-1})|,|h'(x_n)|) \\ 
                   &+ \min(|h''(x_0)|, 2|h''(x_1)|,\dots, 2|h''(x_{n-1})|, |h''(x_n)|).
  \end{split}
\end{equation*}
This is bounded from below by 
\begin{equation*}
  \begin{split}
   \min(|h'(x_0)|+|h''(x_0)|,& 2|h'(x_1)|+2|h''(x_1)|,\dots,\\
   &2|h'(x_{n-1})|+2|h''(x_{n-1})|, |h'(x_n)|+|h''(x_n)|) \geq K.
  \end{split}
\end{equation*}
Therefore we get $M(g_n)(x)\geq K$ for all $n\geq 1$ and $x\in (S^1)^{n+1}$.

We define a open set $\mathcal{V}\subset C^\infty(S^1\times S^1)$ as the set of 
$f \in C^\infty(S^1\times S^1)$ satisfying 
$2\norm{\nabla (f-g)}_\infty + 5\norm{\nabla^2 (f-g)}_{\infty} < K$.
We will prove that for all $f\in \mathcal{V}$ and $n\geq 1$
the functions $f_n: (S^1)^{n+1}\to \mathbb{R}$ are Morse functions and 
\begin{equation} \label{eq: the number of critical points are invariant}
 \#\mathrm{Cr}(f_n) = \#\mathrm{Cr}(g_n) =  (\#\mathrm{Cr}(h))^{n+1}.
\end{equation}

Let $f\in \mathcal{V}$.
By using (\ref{eq: hessian estimate})
\[ \left|\frac{\partial f_n}{\partial x_i}\right| \geq \left|\frac{\partial g_n}{\partial x_i}\right| - 
  2\norm{\nabla (f-g)}_\infty.\]
For any $u\in \mathbb{R}^{n+1}$ with $\|u\|=1$,
\[ \|\hess_x(f_n)u\| \geq \|\hess_x(g_n)u\| - 5\norm{\nabla^2 (f-g)}_\infty.\]
Therefore 
\[ M(f_n)(x) \geq M(g_n)(x) - 2\norm{\nabla (f-g)}_\infty - 5\norm{\nabla^2 (f-g)}_\infty > K-K=0.\]
Hence $M(f_n)(x)>0$ for all $x\in (S^1)^{n+1}$.
This shows that $f_n$ is a Morse function.

For all $t\in [0,1]$, the functions $t g + (1-t)f :S^1\times S^1\to \mathbb{R}$ are 
contained in $\mathcal{V}$.
Since non-degenerate critical points are persistent, this implies the equation 
(\ref{eq: the number of critical points are invariant}).

\vspace{10mm}

\address{ Masayuki Asaoka \endgraf
Department of Mathematics, Kyoto University, Kyoto 606-8502, Japan}

\textit{E-mail address}: \texttt{asaoka@math.kyoto-u.ac.jp}

\vspace{0.5cm}

\address{ Tomohiro Fukaya \endgraf 
Mathematical Institute, Tohoku University, Sendai 980-8578,
Japan}

\textit{E-mail address}: \texttt{tomo@math.tohoku-u.ac.jp}

\vspace{0.5cm}

\address{ Kentaro Mitsui \endgraf
Graduate School of Mathematical Sciences, the University of Tokyo,
Tokyo 153-8914, Japan}

\textit{E-mail address}: \texttt{mitsui@ms.u-tokyo.ac.jp}

\vspace{0.5cm}

\address{ Masaki Tsukamoto \endgraf
Department of Mathematics, Kyoto University, Kyoto 606-8502, Japan}

\textit{E-mail address}: \texttt{tukamoto@math.kyoto-u.ac.jp}


\begin{thebibliography}{99}



\bibitem{Artin-Mazur}
M. Artin, B. Mazur, 
On periodic points, 
Ann. of Math. \textbf{81} (1965), 82-99








\bibitem{Asaoka-Fukaya-Tsukamoto}
M. Asaoka, T. Fukaya, M. Tsukamoto,
Remark on dynamical Morse inequality, 
Proc. Japan Acad. \textbf{87} Ser. A (2011). 178-182




\bibitem{Bertelson}
M. Bertelson,
Topological invariant for discrete group actions,
Lett. Math. Phys. \textbf{62} (2004), 147-156




\bibitem{Bertelson-Gromov}
M. Bertelson, M. Gromov, 
Dynamical Morse entropy,
in: Modern dynamical systems and applications, 27-44,
Cambridge Univ. Press, Cambridge, 2004




\bibitem{Bochnak-Coste-Roy}
J. Bochnak, M. Coste, M-F. Roy, 
Real algebraic geometry,
Translated from the 1987 French original, revised by the authors,
Springer-Verlag, Berlin, 1998



\bibitem{Du}
 P. Duarte,
 Abundance of elliptic isles at conservative bifurcations,
  Dynam. Stability Systems 14 (1999), no. 4, 339--356.


\bibitem{Fukaya-Tsukamoto}
T. Fukaya, M. Tsukamoto,
Asymptotic distribution of critical values,
Geom. Dedicata \textbf{143} (2009), 63-67




\bibitem{Fulton}
W. Fulton,
Intersection theory,
Springer-Verlag, Berlin, 1984




\bibitem{GTS}
S. Gonchenko, D. Turaev, L. Shilnikov, 
 Homoclinic tangencies of arbitrarily high orders
 in conservative and dissipative two-dimensional maps,
 Nonlinearity 20 (2007), no. 2, 241--275.





\bibitem{Hartshorne}
R. Hartshorne, 
Algebraic geometry, 
Springer-Verlag, 
New York-Heidelberg, 1977






\bibitem{Hironaka}
H. Hironaka,
Resolution of singularities of an algebraic variety over a field of characteristic zero,
Ann. of Math. \textbf{79} (1964), 109-326



\bibitem{Kaloshin 1}
V. Y. Kaloshin,
An extension of the Artin-Mazur theorem,
Ann. of Math. \textbf{150} (1999), 729-741



\bibitem{Kaloshin 2}
V. Y. Kaloshin,
Generic diffeomorphisms with superexponential growth of the number of periodic orbits,
Comm. Math. Phys. \textbf{211} (2000), 253-271



\bibitem{KH}
 A.~Katok and B.~Hasselblatt,
 Introduction to the modern theory of dynamical systems,
 Encyclopedia of Mathematics and its Applications,
 54. {\it Cambridge University Press, Cambridge}, 1995.




\bibitem{King}
H. King, 
Approximating submanifolds of real projective spaces by varieties,
Topology \textbf{15} (1976), 81-85



\bibitem{Kucharz}
W. Kucharz,
Transcendental submanifolds of projective space,
Comment. Math. Helv. \textbf{84} (2009), 127-133





\bibitem{Milnor}
J. Milnor, 
Morse theory,
Annals of Mathematics Studies, \textbf{51}, Princeton University Press,
Princeton, 1963




\bibitem{Milnor-real varieties}
J. Milnor, 
On the Betti numbers of real varieties,
Proc. Amer. Math. Soc. \textbf{15} (1964), 275-280



\bibitem{Mo}
J. Moser, 
 On invariant curves of area-preserving mappings of an annulus,
 Nachr. Akad. Wiss. G\"{o}ttingen Math.-Phys. Kl. II 1962 1962 1--20.



\bibitem{Mumford-red book}
D. Mumford, 
The red book of varieties and schemes,
Second expanded edition, 
Lecture Notes in Mathematics, \textbf{1358},
Springer-Verlag, Berlin, 1999




\bibitem{Nash}
J. Nash, 
Real algebraic manifolds,
Ann. of Math. \textbf{56} (1952), 405-421



\bibitem{Ro}
C. Robinson, Dynamical systems.
 Stability, symbolic dynamics, and chaos, Second edition,
 Studies in Advanced Mathematics. CRC Press, Boca Raton, FL, 1999. 




\bibitem{Tognoli}
A. Tognoli,
Su una congettura di Nash, 
Ann. Scuola Norm. Sup. Pisa, Sci. Fis. Mat. \textbf{27} (1973), 167-185














\end{thebibliography}
\end{document}